\documentclass[11pt,reqno]{amsart}
\usepackage{amsbsy,amsfonts,amsmath,amssymb,amscd,amsthm}
\usepackage{mathrsfs,float,color}
\usepackage[mathcal]{euscript}
\usepackage{subfigure,enumitem,graphicx,epstopdf}
\usepackage[a4paper,text={6.1in,8in},centering,top=1.5in]{geometry}
\usepackage{pxfonts,epstopdf}
\usepackage[colorlinks=true,linkcolor=blue,citecolor=green]{hyperref}
\usepackage{srcltx}
\usepackage[margin=20pt,font=small,labelfont=bf,textfont=it]{caption}

\small\normalsize
\theoremstyle{plain}
\newtheorem{mainthm}{Theorem}
\newtheorem{maincor}{Corollary}
\newtheorem{thm}{Theorem}[section]

\newtheorem{lem}[thm]{Lemma}

\theoremstyle{definition}

\newcommand{\eqdef}{\stackrel{\scriptscriptstyle\rm def}{=}}

\begin{document}

\title[
A formula for the local metric pressure]{A formula for the local metric pressure\\
}

\author[M. Carvalho]{Maria Carvalho}
\address{Maria Carvalho\\ Centro de Matem\'{a}tica da Universidade do Porto\\ Rua do
Campo Alegre 687\\ 4169-007 Porto\\ Portugal}
\email{mpcarval@fc.up.pt}

\author[S. P\'erez]{Sebasti\'an A. P\'erez}
\address{Sebasti\'an A. P\'erez\\ Centro de Matem\'{a}tica da Universidade do Porto\\ Rua do
Campo Alegre 687\\ 4169-007 Porto\\ Portugal}
\email{sebastian.opazo@fc.up.pt}


\maketitle


\setcounter{tocdepth}{2}

\begin{abstract}
In this note we present a formula for the local metric pressure that generalizes Brin-Katok result for the metric entropy. As an application, we give a straightforward proof of the fact that non-atomic weak-Gibbs invariant probability measures are equilibrium states. 
\end{abstract}

\section{Introduction}

Let  $(X,d)$ be a compact metric space, $f\colon X \to X$ a continuous transformation and $\varphi \colon X \to \mathbb{R}$ a continuous potential. The \emph{topological pressure} $P_{\text{top}}(f,\varphi)$ of $f$ and $\varphi$ is a topological invariant that generalizes the notion of topological entropy of $f$, one denotes by $h_{\text{top}}(f)$, in the sense that $P_{\text{top}}(f,\varphi) = h_{\text{top}}(f)$ whenever $\varphi\equiv 0$. We refer the reader to \cite{W1981} for precise definitions and properties of these notions. A Borel $f$-invariant probability measure $\mu$ is said to be an \textit{equilibrium state} for $f$ and the potential $\varphi$ if
$$P_{\mu}(f,\varphi) = \sup_{\{\nu \colon \, f_*(\nu)\,=\,\nu\}} \,\big\{\,P_{\nu}(f,\varphi)\,\big\}$$
where the $P_{\nu}(f,\varphi)$ stands for the sum $h_{\nu}(f) + \int \varphi \,d\nu$ and the supremum is taken over all the Borel $f$-invariant probability measures. 
According to the Variational Principle (\cite[Theorem 9.10]{W1981}), the previous least upper bound coincides with the supremum evaluated on the set of ergodic probability measures, and is equal to $P_{\text{top}}(f,\varphi)$. 
We will show how to estimate the metric pressure of any continuous potential, thereby generalizing Brin-Katok formula for the metric entropy \cite{BK1983}. 


\begin{mainthm}\label{thm.BK} 
Let $\mu$ be a Borel non-atomic $f$-invariant probability measure. Then there exists a $\mu$-integrable map $\,x \, \in \, X  \, \, \mapsto \, \,P_{\mu}(x,f,\varphi)$ which is $f$-invariant and satisfies
$$\int P_{\mu}(x,f, \varphi)\,d\mu = P_{\mu}(f, \varphi).$$
If, in addition, $\mu$ is ergodic, then $P_{\mu}(x,f, \varphi) = P_{\mu}(f, \varphi)$ for $\mu$ almost every $x \in X$.
\end{mainthm}

We remark that, when $\varphi \equiv 0$, the map  $\,x \, \in \, X  \, \, \mapsto \, \,P_{\mu}(x,f,\varphi)$ is the  local entropy as defined by Brin and Katok. 


A Borel probability measure $\mu$ is said to be \textit{weak-Gibbs} for the dynamical system $f$ with respect to a potential $\varphi$ if there exists $\varepsilon_0 > 0$ and a subset $\Lambda \subset X$ with full $\mu$-measure such that, for every $0 < \varepsilon < \varepsilon_0$ and every $x \in \Lambda$, there is a sequence of positive constants $\big(\delta_n(\varepsilon, x)\big)_{n \, \in \, \mathbb{N}}$ satisfying
\begin{equation}\label{eq:Gibbs-delta}
\lim_{n \, \to\, + \infty}\,\frac{\log \delta_n(\varepsilon, x)}{n}=0
\end{equation}
and, for every $n \in \mathbb{N}$,
\begin{equation}\label{eq:Gibbs}
\delta_n(\varepsilon, x)^{-1} \leqslant \frac{\mu \big(B^f_n(x,\varepsilon)\big)}
{\exp\big(-P_{\text{top}}(\varphi,f)\,n + S^f_n\varphi(x)\big)}
 \leqslant \delta_n(\varepsilon, x)
\end{equation}
where 
$$B^f_n(x,\varepsilon) = \Big\{y \in X \colon \, d(f^i(x), f^i(y)) < \varepsilon, \quad \forall \, 0 \leqslant i \leqslant n-1\Big\}$$
is the $n$th dynamical ball of $f$ at $x$ with radius $\varepsilon$ 
and $S^f_n\varphi(x)$ stands for the $n$th Birkhoff's sum $\sum_{i=0}^{n-1}\,\varphi(f^i(x))$ at $x$ associated to the dynamics $f$ and the fixed potential $\varphi$.
We say that a weak-Gibbs measure $\mu$ for $f$ with respect to $\varphi$ is \emph{Gibbs} if the sequence $\big(\delta_n(\varepsilon, x)\big)_{n \, \in \, \mathbb{N}}$ is independent of $n$ and $x$.

\begin{maincor}\label{cor.ES} 
Let $f\colon X \to X$ be a continuous map on a compact metric space $(X,d)$ whose topological entropy $h_{\text{top}}(f)$ is finite and which preserves a Borel non-atomic probability measure $\mu$. Consider a continuous potential $\varphi\colon X \to \mathbb{R}$. If $\mu$ is a weak-Gibbs measure for $f$ with respect to $\varphi$, then $\mu$ is an equilibrium state for $f$ and $\varphi$.
\end{maincor}

\section{Proof of Theorem~\ref{thm.BK}}

In this section we will extend Brin-Katok local entropy formula to general continuous potentials
(another generalization may by found in~\cite{ZZC2014}).
 Brin-Katok's result asserts that, given a compact metric space $X$, a continuous map $f:X \to X$ and Borel non-atomic $f$-invariant probability measure $\mu$, there exists a full $\mu$-measure set $\mathcal{BK}\subset X$ such that:
\begin{itemize}
\item [(a)] For every $x \in \mathcal{BK}$,
$$h_{\mu}(x,f)\eqdef \lim_{\varepsilon\to 0} \limsup_{n\to+\infty} \frac{-\log \mu\big(B^f_n(x,\varepsilon)\big)
}{n}=\lim_{\varepsilon\to 0} \liminf_{n\to+\infty} \frac{-\log \mu\big(B^f_n(x,\varepsilon)\big)
}{n}
$$
is well defined.
\medskip
\item [(b)] The map $\,x \in \mathcal{BK} \,\mapsto \,h_{\mu}(x,f)\,\,$ is $f$-invariant.
\medskip
\item [(c)] $\int h_{\mu}(x,f) \,d\mu = h_{\mu}(f)$.
\end{itemize}

\medskip

Having fixed a continuous potential $\varphi$ whose pressure $P_{\text{top}}(f, \varphi)$ finite, consider the Birkhoff's sums $S^f_n\varphi$ and, for $x \in X$, define \emph{the local pressure of $\varphi$ at $x$} by
\begin{eqnarray*}
P_{\mu}(x,f,\varphi)&=&\lim_{\varepsilon\,\to\, 0} \,\limsup_{n\,\to\,+\infty} \,\frac{-\log \mu \big(B^f_n(x,\varepsilon)\big)+ S^f_n\varphi(x)}{n} 
=\lim_{\varepsilon\,\to\, 0} \,\liminf_{n\,\to\,+\infty}\, \frac{-\log \mu \big(B^f_n(x,\varepsilon)\big)+ S^f_n\varphi(x)}{n}
\end{eqnarray*}
if these limits exist and are equal.

\begin{lem}\label{le.BK} The following properties are valid for $\mu$: 
\begin{itemize}
\item [(a)] $P_{\mu}(x,f,\varphi)$ is well defined at $\mu$ almost every $x \in X$.
\medskip
\item [(b)] The map $\,x \, \mapsto \, P_{\mu}(x,f,\varphi)$ is $f$-invariant.
\medskip
\item [(c)] $\int P_{\mu}(x,f, \varphi)\,d\mu = h_{\mu}(f) + \int \varphi \,d\mu$.
\medskip
\item [(d)] If, in addition, $\mu$ is ergodic, then $P_{\mu}(x,f, \varphi) = h_{\mu}(f) + \int \varphi \,d\mu\,\,$ at $\mu$ almost every $x \in X$.
\end{itemize}
\end{lem}

\begin{proof} Birkhoff's Theorem provides a full $\mu$-measure set $\mathcal{B}_\varphi$ and an $f$-invariant map
$$x \in \mathcal{B}_\varphi \quad \mapsto \quad \widetilde{\varphi}(x) = \lim_{n\,\to\, +\infty} \,\frac1{n}\,\sum^{n-1}_{i=0}\,\varphi(f^i(x))$$
satisfying $\int \widetilde{\varphi} \,d\mu =\int \varphi \,d\mu$.
Therefore, for every $x \in \mathcal{BK} \cap \mathcal{B}_\varphi$, we have
\[\begin{split}
&\liminf_{n\,\to\,+\infty} \,\frac{-\log \mu\big(B^f_n(x,\varepsilon)\big)}{n} + \widetilde{\varphi}(x)
= \liminf_{n\,\to\,+\infty}\, \frac{-\log \mu\big(B^f_n(x,\varepsilon)\big) + S^f_n\varphi(x)}{n}\leqslant\\
&\leqslant \limsup_{n\,\to\,+\infty} \,\frac{-\log \mu\big(B^f_n(x,\varepsilon)\big) + S^f_n\varphi(x)}{n} = \limsup_{n\,\to\,+\infty} \,\frac{-\log \mu\big(B^f_n(x,\varepsilon)\big)}{n} + \widetilde{\varphi}(x)
\end{split}\]
Taking the limit when $\varepsilon$ goes to $0$ at the last inequality, we conclude that
$$P_{\mu}(x,f,\varphi) = h_{\mu}(x,f)+\widetilde{\varphi}(x)$$
exists for every $x \in \mathcal{BK}\cap \mathcal{B}_\varphi$. Items (b), (c) and (d) are immediate after (a).
\end{proof}

\section{Proof of Corollary~\ref{cor.ES}}

Firstly recall that, given a compact metric space $X$ and a continuous map $f:X \to X$, the pressure map $P_{\text{top}}(f,.) \colon \,C^0(X,\mathbb{R}) \, \to \, \mathbb{R}\cup \{+ \infty\}$, defined on the space $C^0(X,\mathbb{R})$ of continuous potentials, is either finite valued or constantly $+ \infty$ (cf. \cite[\S 9.2]{W1981}).

Consider a Gibbs measure $\mu$ for the dynamics $f$ and a continuous potential $\varphi$, and gather the corresponding $\varepsilon_0$, $\Lambda$ and $\big(\delta_n(\varepsilon, x)\big)_{n \, \in \, \mathbb{N}}$ satisfying equations \eqref{eq:Gibbs-delta} and \eqref{eq:Gibbs} for every $x\in \Lambda$ and every $n \in \mathbb{N}$.
As we are assuming that $h_{\text{top}}(f) < +\infty$, we know that $P_{\text{top}}(f, \varphi)$ is finite. Rewriting \eqref{eq:Gibbs}, we obtain, for every $x \in \mathcal{BK} \cap \mathcal{B}_\varphi \cap \Lambda$, 
\begin{equation*}
P_{\text{top}}(f,\varphi)-\frac{\log \delta_n(\varepsilon, x)}{n} \leqslant  \frac{ -\log \mu\big(B^f_n(x,\varepsilon)\big)+S^f_n\varphi(x)}{n}
 \leqslant P_{\text{top}}(\varphi,f)-\frac{\log\delta_n(\varepsilon, x)^{-1}}{n}.
\end{equation*}
Taking $\limsup_{n\to+\infty}$ (or $\liminf_{n\to+\infty}$) and afterwards the limit as $\varepsilon$ goes to $0$, we get 
$$P_{\mu}(x,f,\varphi)=P_{\text{top}}(f,\varphi)\quad \quad \mbox{at} \,\, \mu-a.e. \,\, x \in X.$$
Thus, applying Lemma~\ref{le.BK}, we conclude that
$$h_{\mu}(f) + \int \varphi \,d\mu = \int P_{\mu}(x,f,\varphi)\,d\mu = \int P_{\text{top}}(f,\varphi)\,d\mu = P_{\text{top}}(f,\varphi).$$
Therefore $\mu$ is an equilibrium state for $f$ with respect to $\varphi$.

\section{Open questions}

Is a Gibbs measure for a dynamical system $f$ with respect to a potential $\varphi$ always $f$-invariant? We may also wonder whether the existence of a Gibbs measure for $f$ and $\varphi$ prompts the existence of an equilibrium state for $\varphi$. Or we may ask under what additional conditions, other than $f$-invariance, is a Gibbs measure for $f$ and $\varphi$ an equilibrium state for these dynamics and potential, or 
else an equilibrium state for another natural potential somehow related to $\varphi$. For instance, under a stronger definition of the Gibbs property and assuming that $f$ is a homeomorphism satisfying expansiveness and specification, the answer is positive (cf. \cite{HR1992}). As far as we know, these are still open questions.


\section*{Acknowledgements}
MC and SP were partially supported by CMUP (UID/MAT/00144/2019), which is funded by FCT with national (MCTES) and European structural funds through the programs FEDER, under the partnership
agreement PT2020. SP also acknowledges financial support from a postdoctoral grant of the project PTDC/MAT-CAL/3884/2014

\end{document}